\documentclass[12pt]{amsart}

\usepackage{amssymb}
\usepackage{hyperref}
\hypersetup{colorlinks,%
citecolor=red,%
filecolor=red,%
linkcolor=red,%
urlcolor=red}

\theoremstyle{plain}
\newtheorem{theorem}{Theorem}[section]
\newtheorem{lemma}[theorem]{Lemma}
\newtheorem{corollary}[theorem]{Corollary}
\newtheorem{proposition}[theorem]{Proposition}
\theoremstyle{definition}

\newtheorem{example}[theorem]{Example}
\theoremstyle{remark}

\begin{document}

\author{A. N. Abyzov}
\address{Abyzov Adel Nailevich, Chair of Algebra and Mathematical Logic, Kazan (Volga Region) Federal Univesrity, 18 Kremlyovskaya str., Kazan, 420008 Russia}
\email{aabyzov@ksu.ru, Adel.Abyzov@ksu.ru}

% second author
\author{T. H. N. Nhan}
\address{Tran Hoai Ngoc Nhan, Department of IT and Mathematics Teacher Training, Dong Thap University, Vietnam}
\curraddr{Chair of Algebra and Mathematical Logic, Kazan (Volga Region) Federal Univesrity, 18 Kremlyovskaya str., Kazan, 420008 Russia}
\email{tranhoaingocnhan@gmail.com}

\title[CS-Rickart modules]{CS-Rickart modules}

\begin{abstract}
In this paper, we introduce and study the concept of CS-Rickart modules, that is a module analogue of the concept
of ACS rings. A ring $R$ is called a right weakly semihereditary ring if every its  finitly generated right ideal is of the form $P\oplus S,$ where $P_R$ is a projective module and $S_R$ is a singular module. We describe the ring $R$ over
which $\mathrm{Mat}_n (R)$ is a right ACS ring for any
$n \in \mathbb {N}$. We show that every finitely generated projective right $R$-module will to be a CS-Rickart module, is precisely when $R$ is a right weakly semihereditary ring. Also, we prove that if $R$ is a right weakly semihereditary ring, then  every finitely generated submodule of a projective right $R$-module has the form
$P_{1}\oplus \ldots\oplus  P_{n}\oplus S$, where every $P_{1}, \ldots,  P_{n}$ is a projective module which is isomorphic to a submodule of $R_{R}$, and $S_R$ is a singular module.
As corollaries we obtain some well-known properties of Rickart modules and semihereditary rings.
\end{abstract}

\subjclass[2010]{ Primary 16D10; Secondary 16D40, 16D80.}
\keywords{CS-Rickart modules, Rickart modules, ACS rings, semihereditary rings.}
\date{\today}

\maketitle

\section{Introduction}

Throughout this paper, all rings are assumed to be associated with a nonzero unity element,
and all modules are assumed to be unitary right modules. The notations
$N \unlhd M$, or $N \ll M$ mean that $N$ is
an \textit{essential} (or \textit{large}) submodule of $M$, or $N$ is a \textit{superfluous} (or \textit{small}) submodule of $M$, respectively. The largest
singular submodule of $M$ will be denoted by $Z(M)$.

The concept of p.p. rings was first introduced by Hattori in 1960 (see [\ref{AH}]), and
further studied by many authors (see, for example, [\ref{pp2}, \ref{pp1}, \ref{pp3}]). A ring $R$ is called a right \textit{p.p. ring} (or a \textit{Rickart ring}) if every principal right ideal of $R$ is to be projective, or equivalently, if the right annihilator of each element
of $R$ is generated by an idempotent as a right ideal. According to [\ref{GL1}], the notion of Rickart rings was generalized to a module theoretic version. A right $R$-module $M$ is
called a \textit{Rickart module} if for every
$\varphi \in S=\mathrm{End}_{R}(M)$ then
$\mathrm{Ker}\varphi = eM$ for
some $e^2=e\in S$.  The notion of dual Rickart modules was introduced in [\ref{GL2}].
A module $M$ is called a \textit{d-Rickart module} (or a \textit{dual Rickart module}) if for every
$\varphi \in S=\mathrm{End}_{R}(M)$ then
$\mathrm{Im}\varphi = eM$ for some $e^2=e\in S$. Many characterizations
of Rickart modules and d-Rickart modules are given in [\ref{GL1}, \ref{GL2}].
Article
[\ref{GL3}] has described the class of rings over which each finitely
generated projective right module as a Rickart module. Moreover, the authors also studied the structure of rings over which
direct sums of Rickart modules are also Rickart modules.

A ring $R$ is called a right \textit{ACS ring} if the right annihilator of every element of $R$ is an essential submodule of a direct summand of $R_R$ (see [\ref{WJ}, \ref{WK}]). The concept of ACS rings has been studied in
articles [\ref{WJ}, \ref{WK}, \ref{QZ}], is a proper generalization of Rickart rings. It is known that $R$ is a right ACS ring which is also a right $C_2$-ring if and only if $R$ is semiregular and $J(R) = Z_r(R)$ (see [\ref{WK}, Theorem 2.4]).

Recall that a module $M$ is called an \textit{SIP module} (respectively, \textit{SSP
module}) if the intersection (or the sum) of any two
direct summands of $M$ is also a direct summand of
$M$. Modules having the SIP or the SSP have been studied by many authors (see [\ref{Ga}, \ref{HS},  \ref{Wi}]). It is shown
that every Rickart module has the SIP and every d-Rickart module has the SSP (see [\ref{GL1}, Proposition 2.16] and [\ref{GL3}, Proposition 2.11]).

A module $M$ is called an \textit{SIP-CS module} if
the intersection of any two direct summands of $M$
is essential in a direct summand of $M$.
 We say that a submodule $N$ of a module $M$ \textit{lies above a direct summand} of $M$ if there is a decomposition $M=N_1\oplus
N_2$ such that $N_1 \subset N$ and $N_2\cap N$ small in $N_2$.
A module $M$ is called an \textit{SSP-d-CS module} if the sum of any
two direct summands of $M$ lies above a direct summand
of $M$. While SIP-CS modules are proper generalizations of both SIP modules and CS modules (see [\ref{s1}, \ref{s2}]), SSP-d-CS are proper generalizations of both SSP modules and d-CS modules (see [\ref{s3}).

A module is called a \textit{CS module} if every its submodule is essential in a
direct summand.
Dually, a module is called a \textit{d-CS module} (or a \textit{lifting module}) if every its submodule  lies above a direct summand(see [\ref{JC}, \ref{ex}]).

In  this  paper,  we
introduce and study the notion of CS-Rickart modules, that is a module analogue of the notion
of ACS rings.  A module $M$ is called a \textit{CS-Rickart module} if $\mathrm{Ker}\varphi$
is essential in a direct summand of $M$ for every $\varphi \in S=\mathrm{End}_{R}(M)$.
Dually, a module $M$ is called a \textit{d-CS-Rickart module} if
$\mathrm{Im}\varphi$ lies above a direct summand of $M$ for every $\varphi \in S=\mathrm{End}_{R}(M)$. In \cite{AN} CS-Rickart modules were first introduced and some of the results of
this paper have been presented without proof as a brief
communication.

In Section 2, we provide some characterizations and investigate their properties. We show that any direct
summand of a CS-Rickart module or CS-Rickart module inherits the property (see Lemma \ref{1}), while this is not so for direct sums (see Example \ref{e2}). We show that the class of rings over which every right module is CS-Rickart,
is precisely the class of rings over which every right module is d-CS-Rickart (see Lemma \ref{a2}). We show that every CS-Rickart module has the SIP-CS and every d-CS-Rickart module has the SSP-d-CS.
Also we prove that, the right ACS and the
right essentially Baer properties coincide for a ring with the minimum condition on
right annihilators (see Proposition \ref{4b}).

In Section 3, we establish connections between the CS-Rickart
property of a module and its d-CS-Rickart property (see Theorem \ref{9}) which is a generalization of [\ref{WK}, Theorem 2.4].
We also give conditions for a finitely generated projective module to be a CS-Rickart module which is also a $C_2$ module (see Theorem \ref{11}, Theorem \ref{12}). We conclude this section with conditions which allow a direct sum of
CS-Rickart modules (respectively, d-CS-Rickart modules) to be CS-Rickart (or d-CS-Rickart) (see Theorem \ref{9a}, Theorem \ref{9b}).

We call a module $M_R$ is a \textit{weakly semihereditary module} if every finitely generated submodule of $M$ is of the form $P\oplus S$, where $P_R$ is a projective module and $S_R$ is a singular module. A ring $R$ is called a \textit{right weakly semihereditary ring} if the module $R_{R}$ is weakly semihereditary. In Section 4, it is shown that every finitely generated projective right $R$-module will to be a CS-Rickart module, is precisely when $R$ is a right weakly semihereditary ring. Also we prove that $R$ is a right weakly semihereditary ring if and only if every finitely generated submodule of a projective right $R$-module has the form
$P_{1}\oplus \ldots\oplus  P_{n}\oplus S$, where every $P_{1}, \ldots,  P_{n}$ is a projective module which is isomorphic to a submodule of $R_{R}$  and $S$ is a singular module.

Our results in this paper also give uniform approaches to the works in
[\ref{GL1}], [\ref{GL3}], [\ref{GL2}], [\ref{WK}], [\ref{QZ}].
%[\ref{GL3}, Theorem 2.29], [\ref{GL3}, Theorem 3.6], [\ref{GL2}, Theorem 5.11] and [\ref{WK}, Theorem 2.4].
\section{Preliminary results}
The following lemmas can be verified directly.
They show that every direct summand of a CS-Rickart module (or d-CS-Rickart module) is inherited the property; direct sums of CS-Rickart modules (or d-CS-Rickart modules), submodules of CS-Rickart modules are inherited the property by some conditions.

%Their proofs follow on similar lines to that of [\ref{GL1}, Theorem 2.7],  [\ref{GL3}, Proposition 2.8], [\ref{GL2}, Proposition 2.34] and [\ref{GL3}, Proposition 5.14].

\begin{lemma}\label{1}
The following implications hold:
\begin{itemize}
\item[(1)] Every direct summand of a CS-Rickart module is a CS-Rickart module.

\item[(2)] Every direct summand of a d-CS-Rickart module is a d-CS-Rickart module.
\end{itemize}
\end{lemma}

\begin{lemma}
Let $M =\bigoplus_{i\in I}M_i$ and $M_i$ is a fully invariant
submodule of $M$ for every $i\in I$. Then the following implications hold:
\begin{itemize}
\item[(1)]  If $I$ is an arbitrary index set then $M$ is a CS-Rickart module if and only if each $M_i$ is a CS-Rickart module for every $i\in I$.
\item[(2)] If $I=\{1, \ldots, n\}$ then $M$ is a d-CS-Rickart module if and only if each $M_i$ is a d-CS-Rickart module for every $i\in I$.
\end{itemize}
\end{lemma}

\begin{lemma}
Let $M$ be a CS-Rickart module, and let $N$ be a fully invariant
submodule of $M$. If every endomorphism $\varphi\in
\mathrm{End}_R(N)$ can be extended to an endomorphism
$\overline{\varphi}\in\mathrm{End}_R(M)$, then $N$ is a CS-Rickart
module.
\end{lemma}

% \begin{proof}
% Let $\varphi\in \mathrm{End}_R(N)$ be arbitrary. By assumption,
% there exists $\overline{\varphi}\in\mathrm{End}_R(M)$ such that
% $\overline{\varphi}|_N = \varphi$.
% Since $M$ is CS-Rickart,
% $\mathrm{Ker}\overline{\varphi} \unlhd eM$ for some $e = e^2 \in
% \mathrm{End}_R(M)$. Then
% $\mathrm{Ker}\varphi \mathrm{Ker}\overline{\varphi}\cap N\unlhd eM\cap N=
% eN$. Therefore, $N$ is a CS-Rickart module.
% \end{proof}

While every direct summand of CS-Rickart module inherits the property,
this is not so for direct sums (see Example \ref{e2}). We give conditions which allow direct sums of
CS-Rickart modules to be CS-Rickart.

\begin{lemma}\label{2}
Let $A$ be an uniform hereditary right  $R$-module, $B$ be a singular uniform Artinian right $R$-module. Then the following statements hold:

\begin{itemize}
 \item [(1)] $A \oplus B $ is a CS-Rickart module.

 \item [(2)] If $B$ is not an $A$-injective module, then $A \oplus B$ is not a CS module.
 \end{itemize}
\end{lemma}

\begin{proof}(1) Let $f \in \mathrm {End}_{R}(A \oplus B)$ is an injective homomorphism and $\pi$ is the projection from $A \oplus B$ to $A$.
Since $A$ is a hereditary module and $B$ is a singular module,
$\pi f(B) = 0$ and so $\mathrm{Ker}(\pi f) = A_{0} \oplus B$, where
$A_{0}$ is a direct summand of $A$. Since $A$  is an
indecomposable module, either $A_ {0} = A$ or $A_{0}= 0$. If $ A_{0} = 0$ then
$\mathrm{Ker}f = \mathrm{Ker}f_{\mid B} $ and so $\mathrm{Ker}f \unlhd B$.
Otherwise, $A_ {0} = A$. Then the homomorphism $f$ induces a homomorphism from
$A$ to $B$ and so $\mathrm{Ker}f_ {\mid
A} \unlhd A$ by the singularity of $B$. If
$\mathrm{Ker} f_{\mid B} \neq 0$ then $\mathrm{Ker}f \unlhd
A \oplus B$. If $\mathrm{Ker} f_{\mid B} = 0 $ then $f$ induces an isomorphism from
$B$ to $B$ since $B$ is an Artinian
module. Therefore, for any element $a \in A $ there exists
$b \in B$ such that $f(a) = f(b) $ and so
$f(a-b) = 0$. Thus $\pi \mathrm{Ker}f = A$, $\mathrm {Ker} f \cap
B = 0$ and hence $\mathrm{Ker}f \oplus B = A \oplus B$.

(2) Since $B$ is not an $A$-injective, then there exists
a submodule $A_{0}$ of $A$ and a homomorphism $f: A_{0} \rightarrow
A$ which cannot be extended to a homomorphism from $A$ to
$B$. If the module $A \oplus B$ is a CS module, $N = \{a + f(a) \mid
a \in A_{0} \} \unlhd e(A \oplus B)$ for
some $e = e^2 \in \mathrm{End}_{R}(A \oplus B)$. Since $Z(e(A \oplus
B)) = 0$, $Z (A \oplus B) = B \subset (1 -e) (A \oplus B) $ and
consequently $B = (1 -e) (A \oplus B)$. As a result $A \oplus
B = e(A \oplus B) \oplus B$. This leads to
$e(A \oplus B) = \{a + f'(a) \mid a \in A \}$, $f'_{\mid A_{0}} = f$ for
some homomorphism $f': A \rightarrow B$.
This contradiction shows that the module $A \oplus B$ is not
a CS module.
\end{proof}

\begin{example}
If $R$ is a Dedekind domain and $P$ is a nonzero prime ideal of $R$,
then it is deduced from the previous lemma that $R$-module $R \oplus
(R / P^n)$, where $n$ is a natural number, is a CS-Rickart module
and is a direct sum of CS modules but is not a CS module.
\end{example}

We now describe the class of rings over
which every module is a CS-Rickart module.
\begin{lemma}\label{a2}
The following conditions are equivalent for a ring $R$:
\begin{itemize}
\item[(1)]  Every right $R$-module is d-CS-Rickart.

\item[(2)]  Every right $R$-module is CS-Rickart.

\item[(3)]  $R$ is a left and right Artinian serial ring with $(J(R))^2 = 0$.
\end{itemize}
\end{lemma}

\begin{proof}
The implications $(3) \Rightarrow (1)$ and $(3) \Rightarrow (2)$ are concluded from [\ref{VanP92}].

$(1) \Rightarrow (3)$. According to [\ref{VanP92}], it suffices to show that every right module over the ring $R$ is a lifting module. Let $M$ be an arbitrary right $R$-module and $N$ be a submodule of $M$. Consider the module $N \oplus M$ and the homomorphism $f \in \mathrm{End}(N \oplus M)$ is defined by $f(n, m)= (0,  n)$. Then $f(N \oplus M) = (0, N)$ lies above a direct summand of $N \oplus M$. Consequently, the module $N$ lies above a direct summand of $M$.

$(2) \Rightarrow (3)$. According to [\ref{VanP92}], it suffices to show that every right module over $R$ is a CS module. Let $M$ be an arbitrary right $R$-module and $N$ be a submodule of $M$. Consider the module $M \oplus (M /N) $ and the homomorphism $f \in \mathrm{End}(M \oplus (M /N))$ is defined by $f(m, n)=(0, m + N)$. Then for some idempotent $e \in \mathrm{End}(M \oplus (M / N))$, the module $e (M \oplus (M / N))$ is an essential extension of the module $\mathrm{Ker}f = (N, M/N)$. Consequently,  the direct summand $e (M \oplus (M / N)) \cap (M, 0)$ of $(M, 0)$ is an essential extension of $(N, 0)$.
\end{proof}

It is easy to see that a right Rickart module (respectively d-Rickart module) is a right CS-Rickart module (or d-CS-Rickart module), but the converse is not true, in general. A module $M$ is called a \textit{$\mathcal{K}$-nonsingular module} if, for every essential submodule
$N$ of $M$ and every homomorphism $\varphi \in \mathrm{End}_R (M)$, from the equation $\varphi(N) = 0$ implies that
$\varphi = 0$. Dually, a module $M$ is
called a \textit{$\mathcal{T}$-noncosingular module} if, for every homomorphism $0\not=\varphi \in \mathrm{End}_R (M)$, $\mathrm{Im}\varphi$
is not small in $M$. We now establish a connection between CS-Rickart modules and Rickart modules. The following two lemmas can be verified directly. % Their proofs are similar to [\ref{B1}, Lemma 2.14] and [\ref{DB}, Lemma 2.11].
\begin{lemma} \label{3} For a right $R$-module $M$, the following conditions are equivalent:
\begin{itemize}
\item[(1)] $M$ is a $\mathcal{K}$-nonsingular CS-Rickart module.
\item[(2)] $M$ is a Rickart module.
\end{itemize}
\end{lemma}

\begin{lemma}   For a right $R$-module $M$, the following conditions are
equivalent:
\begin{itemize}
\item[(1)]$M$ is a $\mathcal{T}$-noncosingular d-CS-Rickart module.
\item[(2)]$M$ is a d-Rickart module.
\end{itemize}
\end{lemma}

It is known
that every Rickart module has the SIP and every d-Rickart module has the SSP. Our next two propositions show that every CS-Rickart module has the SIP-CS and every d-CS-Rickart module has the SSP-d-CS.
\begin{proposition}\label{4}
Let $M$ be a CS-Rickart module and $S= \mathrm{End}_R(M)$. Then the following statements hold:

\begin{itemize}
\item[(1)] If $A = eM $, $B = fM$ for some $e^2 = e\in S$ and $f^2 = f\in S$, then there exists
$g^2 = g \in S$ such that $eM \cap fM \unlhd gM $.

\item[(2)] If $A \unlhd eM $, $B \unlhd fM$ for some $e^2 = e\in S$ and $f^2 = f\in S$, then
there exists $g^ 2 = g \in S $ such that $ A \cap B \unlhd gM$.

\item[(3)] For arbitrary $\varphi_1, ..., \varphi_n \in S$, there exists an idempotent $e \in S$ such that  $\bigcap_{1\leqslant i \leqslant n}\mathrm{Ker}\varphi_i \unlhd eM$.

\item[(4)] $M$ is an SIP-CS module .
\end{itemize}
\end{proposition}
\begin{proof}
(1) Since $(1-e) M \subset \mathrm{Ker}(1-f)e$, this implies that
$\mathrm{Ker} (1-f)e = eM \cap \mathrm{Ker}(1-f)e \oplus
(1-e)M = eM \cap \mathrm{Ker}(1-f) \oplus (1-e)M $. There exists $h^2 = h \in S $ such that
$\mathrm{Ker}(1-f) e \unlhd hM$ due to $M$ is a
CS-Rickart module. Since $(1-e) M \subset hM$, it follows that
$hM = hM \cap (eM \oplus (1-e) M) = hM \cap eM \oplus (1-e)M$. Thus
$eM \cap \mathrm{Ker}(1-f) \oplus (1-e) M \unlhd hM \cap eM \oplus (1-e) M$
and so $eM \cap fM = eM \cap \mathrm{Ker}(1-f) \unlhd hM \cap eM $.
Since $eM \cap hM $ is a direct summand of $hM$, $eM \cap hM$ is a direct
summand of $ M $.

(2) and (4) follow from (1) while  (3) follows from (2).
\end{proof}

Proposition 2.2 of [\ref{WJ}] gives a proposition of right ACS and left $C_3$
rings. It follows from the previous proposition that we can remove the assumption $R$ is a left $C_3$ ring.
\begin{corollary}
Let $R$ be a right ACS. Then for any
$x_1,\ldots,x_n$ of $R$, there exists an idempotent $e\in R$ such that $r_R(x_1,\ldots,x_n)\unlhd eR$.
\end{corollary}

We recall that a ring $R$ is called a ring with the \emph{minimum condition on
right annihilators} if $R$ does not contain an infinite descending chain of right ideals that are right annihilators of subsets of $R$. A ring $R$ is called a \emph{right essentially Baer ring} if the right annihilator of any nonempty subset of $R$ is essential in a direct summand (see [\ref{Bi}, Definition 8.1.1]). It follows from the previous proposition that, for a ring with the minimum condition on
right annihilators, the right ACS and the
right essentially Baer properties coincide.

\begin{corollary}\label{4b}
Let $R$ be a ring with the minimum condition on right annihilators. Then the following conditions are equivalent:
\begin{itemize}
 \item [(1)] $R$ is a right ACS ring.

 \item [(2)] $R$ is a right essentially Baer ring.
 \end{itemize}
\end{corollary}

\begin{proposition}\label{6}
Let $M$ be a d-CS-Rickart module and and $S= \mathrm{End}_R(M)$. Then the following statements hold:

\begin{itemize}
 \item[(1)] If $A = eM$, $B = fM$ for some $e^2 = e\in S$ and $f^2 = f\in S$, then there exists
$g^2 = g \in S$ such that $A + B$ lies above $gM$.

 \item[(2)] If $A$ lies above $eM$, $B$ lies above $fM$ for some $e^2 = e\in S$ and
$f^2 = f\in S$, then there exists $g^2 = g \in S$ such that $A + B$ lies above $gM$.

 \item[(3)] For arbitrary homomorphisms $\varphi_1, ..., \varphi_n \in S$, there exists an
idempotent $e\in S$ such that $\sum_{i = 1}^n \mathrm{Im}\varphi_i$
lies above $eM$.

 \item[(4)] $M$ is an SSP-d-CS module.
 \end{itemize}

\end{proposition}
\begin{proof}
(1) Since $M$ is a d-CS-Rickart module, there exists
$h^2 = h \in S$ such that $hM \subset (1-e)fM $ and $(1-h)(1-e)fM \ll M$.
Since $ eM + fM = eM \oplus (1-e) fM$, it follows that $eM + fM = eM \oplus hM \oplus
(1-h)(1-e) fM $. Therefore,  according to [\ref{JC}, 22.1], the module $eM + fM$ lies above
the direct summand $eM \oplus hM$ of $M$.

(2) and (4) follow from (1), and (3) follows from (2).
\end{proof}

We give connections between CS-Rickart modules and CS modules; d-CS-Rickart modules and d-CS modules.
\begin{corollary} The following implications hold:
\begin{itemize}
\item[(1)] If $M$ is a CS-Rickart module and every submodule of $M_R$ is a
right annihilator in $M$ of some finitely generated left ideal of $S=\mathrm{End}_RM$,
then $M$ is a CS module.

\item[(2)] If $M$ is a d-CS-Rickart module and every submodule of $M_R$ is a sum of finitely many homomorphic images of $S=\mathrm{End}_RM$ then $M$ is a d-CS module.
\end{itemize}
\end{corollary}

\begin{proof}
(1) If $N\leq M$ then $N=r_M(\varphi_1, ..., \varphi_n)$ for some $\varphi_1, ..., \varphi_n \in S$. Since $M$ is a CS-Rickart module, by Proposition \ref{4}, there exists an idempotent $e \in S$ such that $r_M(\varphi_1, ..., \varphi_n)\unlhd eM$.

(2) If $N\leq M$ then $N=\sum_{i=1}^n \varphi_i N$ for some $\varphi_1, ..., \varphi_n \in S$. Since $M$ is a d-CS-Rickart module, by Proposition \ref{6}, there exists an idempotent $e \in S$ such that  $\sum_{i=1}^n \varphi_i N$ lies above $eM$.
\end{proof}

\section{CS-Rickart modules and $C_2$ condition}

Let $M_R$ be a module. Every module which is isomorphic to
a submodule of some homomorphic image of a direct sum of copies of $M$ is called
an \textit{$M$-subgenerated module}. The full subcategory of all right $R$-modules consists of all
$M$-subgenerated modules is denoted by $\sigma(M)$. It is called \textit{the Wisbauer category of
the module $M$}.

Let $M$ and $N$ be right $R$-modules. We then define $Z_M(N)$ as \textit{the largest $M$-singular submodule of $N$}, it means that $$Z_M(N) =\sum_{f\in \mathrm{Hom}_R(X, N), \mathrm{Ker} f \unlhd X, X \in\sigma(M)} f (X).$$

To prove the main theorem of this section, we give two useful lemmas.

\begin{lemma}\label{7}
The following statements hold:
\begin{itemize}
\item[(1)] If $M$ is a right $R$-module and $P$ is a projective
module in the category $\sigma (M)$, then, for any submodule $N$ of $P$, the following conditions are equivalent:

\begin{itemize}
\item[(a)] $N$ is an essential submodule of $P$;
\item[(b)] $P/N $ is an $M$-singular module.
\end{itemize}

\item[(2)] If $M$ is a right $R$-module, then every nonzero projective
module in the category $\sigma (M)$ is not an $M$-singular.

\item[(3)] If $R$-module $P$ is a nonzero finitely generated quasi-projective module, then $Z_{P}(P) \neq P$.

\end{itemize}
\end{lemma}

\begin{proof}

(1)  Let $P/N$ be an $M$-singular module and $f: P \rightarrow
P/N$ be the natural homomorphism. There exists an epimorphism $g: A \rightarrow P/N$ for some module $A \in\sigma (M)$ such that $\mathrm{Ker}g\unlhd A$. Since $P$ is a projective module in the
category $\sigma (M)$, there exists a homomorphism $h$ such that $gh = f$. Therefore, $N = h^{-1}(\mathrm{Ker}g)\unlhd P$.

(2) and (3) follow directly from (1).
\end{proof}

\begin{lemma}\label{8}
Let $M$ be a right $R$-module, $P$ be a projective
module in the category $\sigma (M)$ and $N \in \sigma (M)$. Then, for any homomorphism $\varphi \in
\mathrm{Hom}_{R}(P, N)$, the following conditions are equivalent:

\begin{itemize}
    \item[(1)] $\mathrm{Ker}(\varphi)\unlhd eP$ for some $e^2 = e \in
\mathrm{End}_{R}(P)$.

    \item[(2)] $\varphi P = P_{0} \oplus S$, where $P_{0}$ is a projective module in
the category $\sigma (M)$ and $S$ is an $M$-singular module.
\end{itemize}
\end{lemma}
\begin{proof}

$(1) \Rightarrow (2)$. This implication can be verified directly.

$(2) \Rightarrow (1)$. Let $\pi$ be the projection from $P_{0} \oplus S$
onto $P_0$. Since $P_{0}$ is a projective module in
the category $\sigma (M)$, this  suggests that $P = \mathrm{Ker} (\pi \phi) \oplus eP$ for some $e^2 = e \in
 \mathrm{End}_{R} (P)$. Then $S \cong \mathrm{Ker} (\pi \phi) / \mathrm{Ker} (\phi)$ which is deduced  from the previous lemma that $\mathrm{Ker} (\varphi) \unlhd \mathrm{Ker} (\pi \phi)$.
\end{proof}

Let $M$ be a right $R$-module. Denote by $\triangledown (M)$ the set $\{f\in \mathrm{End}_RM \mid \mathrm{Im}f\ll M\}$ and $\vartriangle (M)$ the set $\{f\in \mathrm{End}_RM \mid \mathrm{Ker}f\unlhd M\}$. Now we establish connections between the CS-Rickart
property of a module and its d-CS-Rickart property.

\begin{theorem}\label{9} Let $M$ be a right $R$-module and $P$ be a projective
module in the category $\sigma (M)$. Then the following conditions are equivalent:

\begin{itemize}
 \item [(1)] For any homomorphism $\varphi \in \mathrm{End}_R (P)$, we have
$\varphi (P) = eP \oplus P'$, where $P'$ is an $M$-singular module and $e^2 = e \in \mathrm{End}_R(P)$.
 \item [(2)] $P$ is a CS-Rickart module satisfying $C_{2}$ condition.
 \item [(3)] $P$ is a d-CS-Rickart module satisfying $\vartriangle(P) = \triangledown(P)$.
\end{itemize}
\end{theorem}

\begin{proof}
$(1) \Rightarrow (2)$. The previous lemma implies that the module $P$ is a CS-Rickart module. Let $P_0$ be a submodule of $P$. If $P_0\cong eP$ for some $e^2=e \in \mathrm{End}_R (P)$, then
$P_0$ is a projective module in $\sigma(M)$. Consequently, there exists a homomorphism $f\in \mathrm{End}_R (P)$ such that
$\mathrm{Im}f = P_0$.
Since the condition of (1) and Lemma \ref{7}, $P_0$ is a direct summand of $P$.

$(2) \Rightarrow (1)$. The implication is drawn from the previous lemma.

$(1) \Rightarrow (3)$. The inclusion $\triangledown (P)\subset\vartriangle(P)$ is deduced from
Lemma \ref{7}. First, we prove  $\vartriangle(P) \subset \triangledown (P)$. Let $\phi\in \vartriangle(P)$. Since $\mathrm{Ker} (1 - \phi) = 0$ and
$P$ satisfies the $C_2$ condition, $(1 - \phi) P$ is a direct summand of $P$. Consequently,
homomorphism $1 - \phi$ is a left invertible. Since $\vartriangle (P)$ is an ideal of $\mathrm{End}_R(P)$, this implies that $\vartriangle(P) \subset J (\mathrm{End}_R (P))$.  Therefore, according to [\ref{WB}, 22.2] we have the inclusion $\vartriangle(P) \subset \triangledown(P)$.

We next demonstrate that $P$ is a d-CS-Rickart module. Let $\phi\in \mathrm{End}_R (P)$. The conditions of (1) show that $\phi(P) = eP \oplus P'$ where $P'$ is an $M$-singular module and $e^2 = e \in\mathrm{End}_R (P)$.
Let $\pi$ be the projection from $eP \oplus P'$ onto $P'$. Lemma
\ref{7} follows that $\mathrm{Ker}\pi\phi \unlhd P$. Since $\vartriangle(P) = \triangledown(P)$, this implies that $\pi\phi(P) = P'\ll P$ and so $\phi (P)$ lies above the direct summand $eP$ of $P$.

$(3) \Rightarrow (1)$. Let $\phi\in \mathrm{End}_R (P)$. Since $P$ is a d-CS-Rickart module, $\phi(P) =
eP \oplus (1 - e) \phi (P)$ for some $e^2 = e \in \mathrm{End}_R (P)$ and $(1 - e)\phi (P) \ll P$. Thus,  $(1 - e)\phi\in \vartriangle(P)=\triangledown(P)$ which leads to  $(1 - e) \phi (P)$ is an $M$-singular.
\end{proof}

\begin{corollary} [\text{[}\ref{WK}, Theorem 2.4\text{]}] The following conditions are equivalent for a ring $R$:
\begin{itemize}
 \item[(1)] $R$ is a semiregular ring and $J(R) = Z_r(R)$.
 \item[(2)] The ring $R$ is a right ACS ring which is also a right $C_{2}$ ring.
 \item[(3)] If $T$ is a finitely generated right ideal, then
$T=eR \oplus S$ where $e =e^2 \in R$ and $S$ is a right singular ideal of $R$.
\end{itemize}
\end{corollary}

Next we give a necessary and sufficient condition for a ring over which every finitely generated projective module to be a CS-Rickart module which is also a $C_2$ module.
\begin{theorem}\label{11} The following conditions are equivalent for a ring $R$:
\begin{itemize}
 \item[(1)] $R$ is a semiregular ring and $J(R) = Z_r(R)$;
 \item[(2)] Every finitely generated projective $R$-module is a CS-Rickart module which is also a $C_2$ module.
\end{itemize}
\end{theorem}
\begin{proof} The implication $(2) \Rightarrow (1)$ is deduced from Theorem \ref{9}.

$(1) \Rightarrow (2)$. Let $P$ be a finitely generated projective
module. According to [\ref{WB}, 42.11], it follows that every
finitely generated submodule of $P$ lies above a direct
summand of $P$. The implication is drawn directly from
Theorem \ref{9}.
\end{proof}

\begin{theorem}\label{12}
Let $M$ be a right $R$-module, $P \in \sigma (M)$ and $S = \mathrm{End}_{R}(P)$.
Then, the following conditions are equivalent:
\begin{itemize}
 \item[(1)] For any $\varphi \in S$, there exists $e^2 = e \in
S$ such that $eP \subseteq \varphi P$, $(1-e)\varphi P \subseteq
Z_{M} (P)$.
\item[(2)] For any $\varphi_1, \ldots, \varphi_n \in S$,
there exists $e^2 = e \in S $ such that $eP \subset \varphi_1P + \ldots +
\varphi_n P$ and $(\varphi_1P + \ldots + \varphi_n P) \cap (1 -e) P \subset
Z_{M}(P)$.
\item[(3)] For any $\varphi_1, \ldots, \varphi_n \in S$, $\varphi_1P + \ldots + \varphi_n P = eP \oplus
U$ for some $e^2 = e \in S$ and $U \subset Z_{M}(P)$.
\item[(4)] For any $\varphi \in S $, there exists $e^2 = e \in
S$ such that $\varphi P = eP \oplus U$ where $ U \subset Z_{M}(P)$.
\end{itemize}
If $P$ is a projective module in the category $\sigma (M)$, then these conditions are
equivalent to the condition:
\begin{itemize}
\item[(5)] $P$ is a CS-Rickart module satisfying $C_{ 2}$ condition.
\end{itemize}
\end{theorem}
\begin{proof} The implications $(2)\Rightarrow (3)$, $(3) \Rightarrow (4)$
are obvious. The equivalent $(1)\Leftrightarrow (4)$ holds true from
$(1-e) \varphi P \cong U$. If $P$ is a projective module
in the category $\sigma (M)$, then the equivalent $(1) \Leftrightarrow (5)$
follows from Theorem \ref{9}.

$(1) \Rightarrow (2)$. We prove the implication by induction on $n$. When $k = 1$, the result is clearly understood.
Assume that (2) is true for $k = n$. Let
$\varphi_1, \ldots, \varphi_n, \varphi_{n +1} \in S$, then
there are some idempotents $e_{1}, e_{2} \in S$ such that $e_{1}P
\subset \varphi_1P + \ldots + \varphi_n P$, $(\varphi_1P + \ldots +
\varphi_n P) \cap (1 -e_{1}) P \subset Z_{M} (P)$, $e_{2} P \subset
\varphi_{n +1} P$ and $\varphi_{n +1} P \cap (1 -e_{ 2}) P \subset
Z_{M}(P)$. There is an idempotent $e_{0} \in S $ such that
$ e_{0} P \subset (1-e_{1}) e_{2}P$, $(1-e_{1})e_{2} P= e_{0} P \oplus
(1-e_{0}) (1-e_{1}) e_{2} P$ and $(1-e_{0}) (1-e_{1}) e_{2} P \subset
Z_{M}(P)$. Then for some idempotent $e \in S$, we have $e_{1}P +
e_{2} P = eP \oplus (1 -e_{0}) (1-e_{1})e_{2}P$. Since $(\varphi_1P
+ \ldots + \varphi_n P + \varphi_{n +1} P) / eP = (e_{ 1 } P + (\varphi_1P
+ \ldots + \varphi_n P) \cap (1 -e_{ 1 }) P + e_{ 2 } P + \varphi_{n +1} P \cap
(1 -e_{ 2 }) P) / eP = (e_{ 1 } P + e_{ 2 } P) / eP + ((\varphi_1P + \ldots + \varphi_n
P) \cap (1 -e_{ 1 }) P + eP) / eP + (\varphi_{n +1} P \cap (1 -e_{ 2 }) P + eP) / eP$, it follows that
the module $(\varphi_1P + \ldots + \varphi_n P + \varphi_{n +1} P) / eP $
is an $M$-singular module. Since $(1 -e)(\varphi_1P + \ldots +
\varphi_n P + \varphi_{n +1} P) \cong (\varphi_1P + \ldots + \varphi_n P +
\varphi_{n +1} P) / eP$, this implies that $(1-e) (\varphi_1P + \ldots + \varphi_n P +
\varphi_{n +1} P) \subset Z_{M}(P)$. \end{proof}

A module $M$ is called \textit{relatively CS-Rickart to $N$} if for every $\varphi\in \mathrm{Hom}_R(M, N)$, $\mathrm{Ker}\varphi$ is an essential submodule of a direct summand of $M$. A module $M$ is called \textit{relatively d-CS-Rickart to $N$} if for every $\varphi\in \mathrm{Hom}_R(M, N)$, $\mathrm{Im}\varphi$ lies above a direct summand of $N$.

\begin{theorem}\label{9a}
Let $M=M_{1}\oplus \ldots \oplus M_{n}$  be a $C_2$ right $R$-module. If $M_i$ is relatively CS-Rickart to $M_{j}$ for all $i,j\in \{1,\ldots,n\}$ then $M$ is a CS-Rickart module.
\end{theorem}

\begin{proof}
By [\ref{sm1}, Theorem 22] and [\ref{sm2}, Corollary 3.2]), we have $S=\mathrm{End}_{R}(M)$ is a semiregular ring. On the other hand, according to [\ref{sm1}, Theorem 10] and [\ref{sm2}, Corollary 3.2])  we have the equalities
$$J(S)=[J(\mathrm{Hom}_{R}(M_{i},M_{j})]=[\vartriangle(\mathrm{Hom}_{R}(M_{i},M_{j})]=\vartriangle[(\mathrm{Hom}_{R}(M_{i},M_{j})].$$
It follows that $M$ is a CS-Rickart module by [\ref{WB}, 41.22].
\end{proof}

\begin{corollary} \label{c1}
Let $M =M_1 \oplus \ldots \oplus M_n$ be a right $R$-module of finite length. If $\mathrm{Hom}_{R}(\mathrm{Soc}(M_{i}), \mathrm{Soc} (M_{j})) = 0$ for each pair of different indices $i, j \in \{1, \ldots, n \}$ and $ M_ {i}$ is an indecomposable CS-Rickart module
for every $i \in \{1, \ldots, n \}$, then $M$ is a CS-Rickart module.
\end{corollary}
The following example shows that a direct sum of CS-Rickart modules is not a CS-Rickart
module, in general.

\begin{example}\label{e2}
Let $K$ be a field and
$$R=\left\{
      \left(
               \begin{array}{ccc}
                \alpha_1 & \alpha_2 & \alpha_3 \\
                0 & \alpha_4 & 0 \\
                0 & 0 & \alpha_5
               \end{array}
      \right)\ |\ \alpha_1, \alpha_2, \alpha_3\in K
    \right\}.$$

Consider

$$e=\left(
               \begin{array}{ccc}
                1 & 0 & 0 \\
                0 & 0 & 0 \\
                0 & 0 & 0
               \end{array}
      \right)$$
and
$$S_1=\left\{
      \left(
               \begin{array}{ccc}
                0 & \alpha & 0 \\
                0 & 0 & 0 \\
                0 & 0 & 0
               \end{array}
      \right)\ |\ \alpha\in K
    \right\},
S_2=\left\{
      \left(
               \begin{array}{ccc}
                0 & 0 & \alpha \\
                0 & 0 & 0 \\
                0 & 0 & 0
               \end{array}
      \right)\ |\ \alpha\in K
    \right\}.$$

We note that $S_1$, $S_2$ are simple right $R$-modules and non-isomorphic. It is clear that $eR$ and $eR/S_1$ are CS-Rickart modules. Now consider $M=eR\oplus eR/S_1$.
We prove that $M$ is not CS-Rickart. Let $\varphi: eR\to eR/S_1$ and
$\langle\varphi\rangle=\{a+\varphi(a)\mid a\in eR\}$. Then we have
$\langle\varphi\rangle\oplus eR/S_1=M$ and $\langle\varphi\rangle\cap eR=S_1$. If $M$ is CS-Rickart then $S_1\unlhd \pi M$, for some $\pi^2=\pi\in \textrm{End}_R(M)$. By Krull-Remak-Schmidt theorem, $\pi M\cong eR$ or $\pi M\cong eR/S_1$. But $\pi M\cong eR$ is not hold true because $eR$ is not uniform, so $\pi M\cong eR/S_1$. Whenever this happens, we have a contraction since $\mathrm{Soc}(\pi M)=S_1\not\cong \mathrm{Soc}(eR/S_1)\cong S_2$. Thus $M$ is not a CS-Rickart module. Also, in connection with Corollary \ref{c1}, we note that $\mathrm{Hom}_{R}(\mathrm{Soc}(eR), \mathrm{Soc} (eR/S_1)) \not= 0.$
\end{example}

The following assertion is proved similarly to the previous theorem.

\begin{theorem}\label{9b} Let $M=M_{1}\oplus \ldots \oplus M_{n}$  be a $D_2$ right $R$-module. If $M_i$ relatively d-CS-Rickart to $M_{j}$ for all $i,j\in \{1,\ldots,n\}$ then $M$ is a d-CS-Rickart module.
\end{theorem}

We conclude this section with a complete characterization for a module when its endomorphism ring is a regular ring. The following statement follows directly from [\ref{sm1}, Theorem 29].

\begin{theorem}[\text{[}\ref{GL3}, Theorem 2.29\text{]}, \text{[}\ref{GL2}, Theorem 5.11\text{]}] \label{9c}
Let $M = M_{1} \oplus \ldots \oplus M_{n}$ be a decomposition of the right $R$-module $M$. Then the following conditions are equivalent:

\begin{itemize}
  \item[(1)] $\mathrm{End}_{R}(M)$ is a regular ring.
  \item[(2)] $M_ {i}$ is $M_{j}$-Rickart which is also $M_{j}$-$C_ {2}$ for each pair of indices $1 \leq i, j \leq n$.
  \item [(3)] $M_{i}$ is $M_{j}$-dual Rickart which is also $M_ {j}$-$D_{2}$ module for each pair of indices $ 1 \leq i, j \leq n$.
  \item[(4)] $M$ is a Rickart module which is also a dual Rickart module.
\end{itemize}
\end{theorem}

\section{Weakly semihereditary rings}

The following two statements can be verified directly from the well-known facts of  perfect rings and semiregular rings.

\begin{lemma}
The following conditions are equivalent for a ring $R$:
\begin{itemize}
    \item[(1)] $R$ is a semiregular ring.
    \item[(2)] Every finitely generated projective right $R$-module
is a d-CS-Rickart module.
\end{itemize}
\end{lemma}

\begin{lemma}
The following conditions are equivalent for a ring $R$:
\begin{itemize}
    \item[(1)] $R$ is a right perfect ring.
    \item[(2)] Every projective right $R$-module
is a d-CS-Rickart module.
\end{itemize}
\end{lemma}

We describe the ring $R$ over
which $\mathrm{Mat}_n (R)$ is a right ACS ring for any
$n \in \mathbb {N}$.
\begin {theorem}\label{13}
The following conditions are equivalent for a ring $R$ and a fixed
$n \in \mathbb {N}$:
\begin{itemize}
    \item[(1)] Every $n$-generated projective right $R$-module is a CS-Rickart module.
    \item[(2)] The free $R$-module $R_{R}^{(n)}$ is a CS-Rickart module.
    \item[(3)] $\mathrm{Mat}_n (R)$ is a right ACS ring.
    \item[(4)] Every $n$-generated right ideal of $R$ has the form
$P \oplus S$, where $P$ is a projective $R$-module and $S$ is a singular
right ideal of $R$.
    \item [(5)] $R$-module $R_{R}^{(n)}$ is relatively CS-Rickart to $R_{R}$.
    \item[(6)] Every $n$-generated submodule of $R_{R}^{(n)}$ has the form
$P_{1}\oplus \ldots\oplus  P_{n}\oplus S$, where every $P_{1}, \ldots,  P_{n}$ is a projective module which is isomorphic to a submodule of $R_{R}$  and $S$ is a singular module.
\end{itemize}
\end{theorem}

\begin{proof}
The implications $(1) \Rightarrow (2)$ and $(6) \Rightarrow (2) $  is obvious. The implication
$(2) \Rightarrow (1)$ is deduced from Lemma \ref{1}. The implication
$(2) \Rightarrow (4)$ and the equivalent
$(4) \Leftrightarrow (5)$ are concluded from Lemma \ref{8}. Next we put
$S = \mathrm{End}_{R} (R_{R}^{(n)})$.

$(2)\Rightarrow (3)$. Let $ \phi \in S$. Then $\mathrm {Ker} \phi \unlhd
eR_{R}^{(n)}$ for some idempotent $e \in S$. We will prove that
$r_{S} (\phi) \unlhd eS$. The inclusion $r_{S}(\phi) \subset eS$
is verified directly. Let $g \in eS $ be a nonzero
homomorphism. Since $gR_{R}^{(n)} \cap \mathrm{Ker} \phi \neq 0$, there
is a homomorphism $h \in S$ such that $\phi gh = 0$ and $gh \neq 0$. Hence
$r_{S} (\phi) \unlhd eS$.

$(3) \Rightarrow (2)$. Let $\phi \in S$. Then $r_{S}(\phi)
\unlhd eS$ for some idempotent $ e \in S $. We will show that
$\mathrm{Ker}f \unlhd eR_{R}^{(n)}$. If $m \in \mathrm{Ker}f$ then
for some homomorphism $g \in S$, we have $gR_{R}^{(n)} = mR$ and $\phi
g = 0$. Then $eg = g$ and hence $m \in eR_{R}^{(n)}$. Let $m$
be a nonzero element of $eR_{R}^{(n)}$, then there is a nonzero
homomorphism $f\in S$ such that $fR_{R}^{(n)} \subset mR$. Since
$r_{S}(\phi)\unlhd eS$, we have $\phi fg = 0$ and $fg \neq 0$ for some homomorphism
$g \in S$. Thus $mR \cap
\mathrm{Ker} \phi \neq 0$.

$(5) \Rightarrow (6)$. Let $N$ be an $n$-generated submodule of $R_{R}^{(n)}$. Then there exist homomorphism $f \in S$ such that $f(M)=N.$ Let $\pi_{i}$ be the projection from
$R_{R}^{(n)}$ to its $i$-th component and
$\varepsilon_{i}: R_{R} \rightarrow R_{R}^{(n)}$ be the embedding for every $1 \leq i \leq n $. Since the module $R_{R} ^ {(n)}$
is relatively CS-Rickart to $R_{R}$, there exists an idempotent $e_{1} \in S$ such that $\mathrm{Ker} \pi_{1} f \unlhd e_{1} R_{R}^{(n)}$. This implies that
$\mathrm {Ker} f \subset \mathrm {Ker} \pi_{1} f$. Now consider the homomorphism
$\pi_{2} f_{| e_{1} R_{R}^{(n)}}$. For some idempotent
$e_{ 2 } \in S$ we have the inclusions $\mathrm{Ker}f \subset
\mathrm{Ker} \pi_{2} f_{| e_{ 1 } R_{R}^{(n)}} \unlhd e_{2} R_{R}^{(n)}$ and
$e_{2} R_{R}^{(n)} \subset e_{1} R_{R}^{(n)}$.
Furthermore, we obtain a family of idempotents $e_{1}, \ldots, e_{n} \in S$ such that
$R_{R}^{(n)} =e_{1}R_{R}^{(n)}\oplus P_{1}, e_{1}R_{R}^{(n)}=e_{2}R_{R}^{(n)}\oplus P_{2},\ldots, e_{n-1}R_{R}^{(n)}=e_{n}R_{R}^{(n)}\oplus P_{n},$ where every $P_{1}, \ldots,  P_{n}$ is isomorphic to a submodule of $R_{R},$ and

$\mathrm{Ker}f\subset\mathrm{Ker}\pi_{1}f\unlhd e_{1} R_{R}^{(n)}$,

$\mathrm{Ker}f \subset\mathrm {Ker}\pi_{2}f_{| e_{1} R_{R}^{(n)}}\unlhd
e_{2} R_{R}^{(n)}$,

$\cdots$

$\mathrm{Ker}f\subset\mathrm{Ker}\pi_{n}f_{| e_{n-1} R_{R}^{(n)}}\unlhd
e_{n}R_{R}^{(n)}$,

$e_{n}R_{R}^{(n)}\subset\ldots\subset e_{2} R_{R}^{(n)}\subset
e_{1} R_{R}^{(n)}$.

Then we have $\mathrm{Ker}f\subset\mathrm {Ker}\pi_{1} f \cap
\mathrm{Ker} \pi_{2}f_{| e_{1} R_{R}^{(n)}} \cap \ldots \cap
\mathrm{Ker} \pi_{n} f_{| e_{n-1} R_{R}^{(n)}} \unlhd e_{n} R_{R}^{(n)}$.
If $p \in \mathrm{Ker}\pi_{1} f \cap \mathrm{Ker}\pi_{2}f_{| e_{1} R_{R}^{(n)}} \cap
\ldots \cap \mathrm{Ker} \pi_{n} f_{| e_{n-1} R_{R}^{(n)}}$, then $f(p)=\sum_{1\leq
i \leq n} \varepsilon_{i} \pi_{i}f(p) = 0$. Thus $\mathrm{Ker}f =
\mathrm{Ker} \pi_{1} f\cap \mathrm{Ker} \pi_{2} f_{| e_{1} R_{R}^{(n)}}\cap \ldots \cap
\mathrm{Ker} \pi_{n} f_{| e_{n- 1 } R_{R}^{(n)}}$ and therefore
$\mathrm{Ker}f \unlhd e_{n} R_{R}^{(n)}$. Since $R_{R}^{(n)} =P_{1}\oplus \ldots \oplus P_{n} \oplus e_{n} R_{R}^{(n)}$, implies $N=f(R_{R}^{(n)})\cong P_{1}\oplus \ldots \oplus P_{n} \oplus (e_{n}R_{R}^{(n)} /ker(f))$.
\end {proof}

Now we describe the right weakly semihereditary rings.
\begin{theorem}\label{14}
The following conditions are equivalent for a ring $R$:
\begin{itemize}
 \item[(1)] Every finitely generated projective right $R$-module
is CS-Rickart.
 \item[(2)] The free $R$-module $R_{R}^{(n)}$ is a CS-Rickart
module for every $n \in \mathbb {N}$.
 \item[(3)] $\mathrm {Mat}_n (R)$ is a right ACS ring for
every $n \in \mathbb{N}$.
    \item[(4)] For some positive integer $m$, every finitely generated right
ideal of the ring $\mathrm {Mat}_m (R)$ has the form
$P \oplus S$, where $P$ is a projective $\mathrm {Mat}_m (R)$-module and $S$ is a singular
right ideal of $\mathrm {Mat}_m (R)$.
 \item[(5)] $R$ is a right weakly semihereditary ring.
\item[(6)] Every finitely generated submodule of projective right $R$-module has the form
$P_{1}\oplus \ldots\oplus  P_{n}\oplus S$, where every $P_{1}, \ldots,  P_{n}$ is a projective module which is isomorphic to a submodule of $R_{R}$ and $S$ is a singular module.
\end {itemize}
\end {theorem}
\begin {proof}
The equivalent of (1), (2), (3), (5), (6) is deduced from Theorem \ref{13}.

$(4) \Rightarrow (2)$. Let $R_{R}^{(n_{0})}$ be a finitely
free right $R$-module. Choose a natural number
$k$ such that $km > n_{0}$, Theorem \ref{13} implies that the ring
$M_{k} (M_{m}(R)) \cong M_{km}(R)$ is a right ACS ring.
Theorem \ref{13} and Lemma \ref{1} then infer that the module $R_{R}^{(n_{ 0})}$
is a CS-Rickart module.

$(5) \Rightarrow (4)$ Theorem \ref{13} implies that for every
natural number $k$, the ring $M_{km} (R) \cong M_{k} (M_{m} (R))$
is a right ACS ring. The implication is  drawn from
Theorem \ref{13}.
\end {proof}

We call a ring $R$ is a \textit{right weakly hereditary ring} if every right ideal of $R$ is of the form $P\oplus S$, where $P_R$ is a projective module and $S_R$ is a singular module. From [\ref{BY}, Theorem 3.2.10], right co-H-rings are right weakly hereditary rings. In particular, weakly hereditary rings are Artinian serial rings and QF-rings.
For a right Artinian ring, we have the following corollary.

\begin{corollary} Let $R$ be a right Artinian ring. Then the following conditions are equivalent:
\begin{itemize}
    \item[(1)] $\mathrm {Mat}_n (R)$ is a right essentially Baer ring for every $n \in \mathbb{N}$.
    \item[(2)] $R$ is a right weakly hereditary ring.
\end{itemize}
\end{corollary}

The next corollary is a result of the previous theorem, Lemma \ref{3} and
the fact that $R$ is a right nonsingular ring
if and only if $\mathrm{Mat}_n(R)$ is a right nonsingular ring for all positive integer $n$.

\begin{corollary} [\text{[}\ref{GL3}, Theorem 3.6\text{]}] The following conditions are equivalent for a ring $R$:
\begin{itemize}
    \item[(1)] Every finitely generated projective right $R$-module is a Rikart module.
    \item[(2)] The free $R$-module $R_R^{(n)}$ is a Rickart module for every $n\in \mathbb{N}$.
    \item[(3)] $\mathrm{Mat}_n(R)$ is a right p.p. ring for every $n\in \mathbb{N}$.
    \item[(4)] $R$ is a right semihereditary ring.
    \item[(5)] The ring $\mathrm{Mat}_n(R)$ is a right semihereditary  ring for some natural number $n$.
\end{itemize}
\end{corollary}

\label{lastpage}


\begin{thebibliography}{00}

\bibitem{AN} A. N. Abyzov and T. H. N. Nhan, \emph{CS-Rickart modules}, Russian Mathematics (Iz. VUZ) 58:5 (2014), 48-52.\label{AN}
\bibitem{pp2} E.P. Armendariz, \emph{A note on extensions of Baer and P.P.-rings}, J. Austral.
Math. Soc. 18 (1974), 470-473.\label{pp2}
\bibitem{BY} Y. Baba; K. Oshiro, \emph{Classical Artinian Rings and Related Topics}, World Scientific Publishing Co. Pte. Ltd., 2009.\label{BY}
\bibitem{Bi}  Gary F. Birkenmeier, Jae Keol Park, S.Tariq Rizvi, \emph{Extensions of Rings and Modules}, Springer-Verlag New York Inc.,  2013.\label{Bi}
\bibitem{JC} J. Clark, C. Lomp, N. Vanaja and R. Wisbauer, \emph{Lifting Modules. Supplements and Projectivity in Module Theory}, Frontiers in Mathematics, Birkh\"{a}user, Basel-Boston-Berlin, 2006.\label{JC}
\bibitem{ex} N.V. Dung, D.V. Huynh, P.F. Smith, and R. Wisbauer, \emph{Extending Modules},
Longman Scientific \& Technical, 1994.
\label{ex}
\bibitem{pp1} S. Endo, \emph{Note on p.p. rings}, Nagoya Math. J. 17 (1960), 167-170.\label{pp1}
\bibitem{Ga} J . L. Garcia, \emph{Properties of direct summands of modules}, Communications in Algebra 17:1 (1989), 73-92.\label{Ga}
\bibitem{AH} A. Hattori, \emph{A foundation of torsion theory for modules over general rings},
Nagoya Math. J. 17 (1960), 147-158.
\label{AH}
\bibitem{HS} J . Hausen, \emph{Modules with summand intersection property}, Communications in Algebra 17:1 (1989), 135-148.\label{HS}
\bibitem{pp3} S. J\o ndrup, \emph{p.p. rings and finitely generated flat ideals}, Proc. Amer. Math.
Soc. 28:2 (1971), 431-435.\label{pp3}
\bibitem{s1} F. Karabacak, \emph{On generaliaztions of extending modules}, Kyungpook Math. J. 49 (2009) ,
557-562.\label{s1}
\bibitem{s2} F. Karabacak and A. Tercan, \emph{On modules and matrix rings with SIP-extending}, Taiwanese
J. Math. 11 (2007), 1037-1044.\label{s2}
\bibitem{K}F. Kasch, \emph{Modules and Rings}, Academic Press, London, England, 1982.\label{K}
\bibitem{GL1} G. Lee, S. T. Rizvi, C. S. Roman, \emph{Rickart Modules}, Communications in
Algebra 38:11 (2010) 4005-4027.\label{GL1}
\bibitem{GL3} G. Lee, S. T. Rizvi, C. S. Roman, \emph{Direct sums of Rickart modules}, Journal of Algebra 353:1 (2012) 62-78.\label{GL3}
\bibitem{GL2} G. Lee, S. T. Rizvi, C. S. Roman, \emph{Dual Rickart Modules}, Communications in Algebra 39:11 (2011) 4036-4058.\label{GL2}
\bibitem{WJ} W. Li, J. Chen, \emph{When CF rings are Artinian}, Journal of Algebra and Its Applications 12:4 (2013).\label{WJ}
\bibitem{WK} W. K. Nicholson, M. F. Yousif,
\emph{Weakly continous and $C_2$ rings}, Communications in Algebra
29:6 (2001) 2429-2446.\label{WK}
\bibitem{sm1} W. K. Nicholson, Y. Zhou, \emph{Semiregular Morphisms}, Communications in Algebra, 34: 1 (2006), 219-233\label{sm1}
\bibitem{B1} Rizvi S.T., Roman C.S, \emph{Baer and quasi-Baer modules}, Communications in Algebra 32:1 (2004), 103-123.
\label{B1}
\bibitem{s3} Y. Talebi, Ali Reza Moniri Hamzekolaee, \emph{On SSP-Lifting Modules}, East-West Journal of Mathematics 01 (2013), 1-7. \label{s3}
\bibitem{BD}D.K. T\"{u}t\"{u}nc\"{u}, R. Tribak, \emph{On dual Baer modules}, Glasgow Math. J., 52:2 (2010), 261-269.\label{DB}
\bibitem{VanP92} N. Vanaja, V. M. Purav, \emph{Characterization of generalized uniserial rings
in terms of factor rings}, Communications in Algebra 20:8 (1992), 2253-2270.\label{VanP92}
\bibitem{Wi} G. V. Wilson, \emph{Modules with the summand intersection property}, Communications in Algebra 14:1 (1986),
21-38.\label{Wi}
\bibitem{WB} R. Wisbauer, \emph{Foundations of Module and Ring Theory. A Handbook for Study and Research}, Gordon and Breach Science Publishers, Reading, 1991.\label{WB}
\bibitem{sm2} Yiqiang Zhou, \emph{On (semi)regularity and the total of rings and modules}, Journal of Algebra 322 (2009) 562-578.\label{sm2}
\bibitem{QZ} Q. Zeng, \emph{Some Examples of ACS-Rings}, Vietnam Journal of Mathematics 35:1 (2007) 11-19.\label{QZ}



\end{thebibliography}
\end{document}